\documentclass [preprint]{elsarticle}

\usepackage{tabularx, pstricks}
\usepackage{amssymb} \usepackage{amsfonts} \usepackage{amsmath}
\usepackage{amsthm} \usepackage{epsfig, subfig} 
\usepackage{ amscd, amsxtra, latexsym}
\usepackage{epsfig, graphicx}
\usepackage[all]{xy}
\usepackage{enumerate}

\newtheorem{lemma}{Lemma}[section]
\newtheorem{thm}[lemma]{Theorem}
\newtheorem{prop}[lemma]{Proposition}
\newtheorem{cor}[lemma]{Corollary}

\theoremstyle{definition}
\newtheorem{defn}[lemma]{Definition}

\newtheorem{rem}[lemma]{Remark} 
\newtheorem{conv}[lemma]{Convention} 
\theoremstyle{definition}

\newcommand{\N}{\ensuremath {\mathbb{N}}}
\newcommand{\R} {\ensuremath {\mathbb{R}}}

\newcommand{\G} {\ensuremath {\mathbb{G}}}

\newcommand{\st}{\ensuremath {\, ^* }}

\newcommand{\calP} {\ensuremath {\mathcal{P}}}

\newcommand{\calU} {\ensuremath {\mathcal{U}}}

\begin{document}

\begin{frontmatter}

\title{Separable and tree-like asymptotic cones of groups}
\author{Alessandro Sisto\fnref{curaddr}}
\ead{sisto@maths.ox.ac.uk}
\fntext[curaddr]{Current address: Mathematical Institute, St. Giles 24-29, Oxford OX1 3LB, United Kingdom}

\address{Scuola Normale Superiore, Piazza dei Cavalieri 7, 56127 Pisa, Italy}

\begin{abstract}
Using methods from nonstandard analysis, we will discuss which metric spaces can be realized as asymptotic cones. Applying the results we will find in the context of groups, we will prove that a group with ``a few'' separable asymptotic cones is virtually nilpotent, and we will classify the real trees appearing as asymptotic cones of (not necessarily hyperbolic) groups. 
 
\end{abstract}


\end{frontmatter}

\section{Introduction}
The asymptotic cones of a metric space are obtained ``rescaling the metric by an infinitesimal factor'', in such a way that ``infinitely far away'' points become close, while points which are not far enough are identified.
\par
They have been introduced by Gromov in~\cite{Gr1} in the proof that a (finitely generated) group of polynomial growth is virtually nilpotent. Van den Dries and Wilkie gave a different and more general definition in~\cite{vDW}, where they slightly generalize and simplify the proof of the aforementioned result of Gromov.
\par
Since then, asymptotic cones have been used in several contexts, such as the proof of quasi-isometric rigidity results for cocompact lattices in higher rank semisimple groups (\cite{KlL}), fundamental groups of Haken manifolds (\cite{KaL1},~\cite{KaL2}), relatively hyperbolic groups (\cite{Dr2}) and others. Also, there is a close connection between many quasi-isometric invariants for groups (e.g. growth and order of Dehn functions) and the topology and geometry of the asymptotic cones, see~\cite{Dr1} for a survey.
\par
To define the asymptotic cones formally we will use nonstandard methods, which are powerful tools to formally deal with concepts such as ``infinitesimals'', ``infinite numbers'', ``points infinitely far away'', etc.
\par
We will then use those methods to investigate the following question: which metric spaces can be obtained as an asymptotic cone of another metric space or of a group? Our results rely on the fact that an internal set (i.e. an ultraproduct of a sequence of sets) is either finite or has cardinality at least $2^{\aleph_0}$ (Lemma~\ref{finuncount:lem}), and they often reflect this dichotomy. Perhaps this lemma inspired Gromov (see~\cite{Gr2}) to ask if it is true that an asymptotic cone of a group has either trivial or uncountable fundamental group (however, this is not true, see~\cite{OOS}).
\par
One of the main results we will prove is the following: 

\begin{thm}\label{proper:thm}
\begin{enumerate}
\item
If the separable metric space $X$ is an asymptotic cone, then $X$ is proper.
\item
Suppose that for some $\mu\ll\nu$ and some $p\in\st Y$ each asymptotic cone of the metric space $Y$ with scaling factor $\nu'$ such that $\mu\leq\nu'\leq\nu$ and basepoint $p$ is separable and that $X=C(Y,p,\nu)$ is homogeneous. Then $X$ has finite Hausdorff and Minkowski dimension. 
\end{enumerate}
\end{thm}

For example, a corollary of point $(1)$ is that the separable Hilbert space cannot appear as an asymptotic cone.
\par
Minkowski dimension is relevant to our setting because Point proved in~\cite{Po} that a group with one proper asymptotic cone of finite Minkowski dimension is virtually nilpotent. Combining this result with the theorem above we get the following (Corollary~\ref{sepvirtnilp:cor}).

\begin{thm}
If there exist $\nu_1\ll\nu_2$ such that all the asymptotic cones of the group $G$ with scaling factor $\nu\in[\nu_1,\nu_2]$ are separable, then $G$ is virtually nilpotent.
\end{thm}

Van den Dries an Wilkie asked if a group with \emph{one} proper asymptotic cone needs to be virtually nilpotent (see~\cite[Remark 6.4-(3)]{vDW}). The theorem above does not provide the answer to that question, but notice that ``very few'' of the asymptotic cones of $G$ are involved in the statement.
\par
We will also prove two results about real trees appearing as asymptotic cones (Corollary~\ref{nocount:cor} and Proposition~\ref{isolated:prop}), that will allow us to show that there are 3 possible real trees appearing as the asymptotic cone of a group, up to isometry (Corollary~\ref{3trees:cor}). More precisely:
\begin{thm}
If the real tree $X$ is the asymptotic cone of a group, then it is a point, a line or a tree with valency $2^{\aleph_0}$ at each point.
\end{thm}

Notice that if a group is finitely presented and has one asymptotic cone which is a real tree, then it is hyperbolic (see the appendix of~\cite{OOS} by Kapovich and Kleiner). Therefore, in the case of finitely presented groups, this classification follows from the results in~\cite{DP}. However, there are examples of groups such that just some of their asymptotic cones are real trees (see~\cite{OOS}) and our results can be applied to metric spaces in general.
\par
Finally, we will prove that all proper metric spaces can be realized as asymptotic cones (Theorem~\ref{positive:thm}).
\par
The definition of asymptotic cone we will present is not the one usually found in the literature. In fact, use of nonstandard methods tends to be avoided and a definition based on ultrafilters is usually given, even though the ultrafilters based definition is just a restatement of the nonstandard definition. The author thinks that the nonstandard definition is far more convenient because, besides providing a lighter formalism, it allows to directly apply basic results about nonstandard extensions, particular cases of which ought to be proved in most arguments if the other definition is used. Also, the nonstandard definition is ``philosophically'' closer to the idea of looking at a metric space from infinitely far away, while the other one is closer to the idea of Gromov of convergence of rescaled metric spaces, which is more complicated to ``visualize''.

\section{Basic notation and definitions}

If $X$ is a metric space, $x\in X$ and $r\in\R^+$, we will denote by $B(x,r)$ (resp. $\overline{B}(x,r)$) the open (resp. closed) ball with center $x$ and radius $r$.
\par
Recall that a metric space $X$ is proper if closed balls in $X$ are compact.
\par

A geodesic (parametrized by arc length) in the metric space $X$ is a curve $\gamma:[0,l]\to X$ such that $d(\gamma(t),\gamma(s))=|t-s|$ for each $t,s\in[0,l]$. The metric space $X$ is geodesic if for each $x,y\in X$ there exists a geodesic from $x$ to $y$.

\begin{defn}
A tripod is a geodesic triangle such that each side is contained in the union of the other two sides.
\par
A real tree is a geodesic metric space such that all its geodesic triangles are tripods.
\end{defn}

\begin{conv}
From now on all real trees are implied to be complete metric spaces.
\end{conv}

If $T$ is a real tree and $p\in T$, the valency of $T$ at $p$ is the number of connected components of $T\backslash\{p\}$.
\par
\begin{defn}
A proper homogeneous metric space $X$ has finite Minkowski dimension if, given $x\in X$, $\overline{B}(x,1)$ has the following property. There exists $p>0$ and sequences $\{k_n\}_{n\in \N}$, $\{r_n\}_{n\in\N}$ with $\{r_n\}$ converging to $0$ and $\{k_n r_n^p\}$ bounded such that $\overline{B}(x,1)$ can be covered by $k_n$ balls of radius $r_n$ centered in $\overline{B}(x,1)$.
\end{defn}

\section{Nonstandard extensions}

For the following sections we will need basic results about the theory of nonstandard extensions. The treatment will be rather informal, for a more formal one see for example~\cite{Go}. Let us start with a motivating example. It is quite evident that being allowed to use non-zero infinitesimals (i.e. numbers $x$ different from $0$ such that $|x|<1/n$ for each $n\in\N^+$) would be very helpful in analysis. Unfortunately, $\R$ does not contain infinitesimals. The idea is therefore to find an extension of $\R$, denoted by $\st\R$, which contains infinitesimals. Let us construct such an extension.
\par
\begin{defn}
Let $I$ be any infinite set. A filter $\calU\subseteq \calP(I)$ on $I$ is a collection of subsets of $I$ such that for each $A,B\subseteq I$
\begin{enumerate}
\item
If $A$ is finite, $A\notin\calU$ (in particular $\emptyset\notin\calU$),
\item
$A,B\in\calU\Rightarrow A\cap B\in\calU$,
\item
$A\in\calU,B\supseteq A\Rightarrow B\in\calU$.
\par
An ultrafilter is a filter satisfying the further property:
\item
$A\notin\calU\Rightarrow A^c\in\calU$.
\end{enumerate}

\end{defn}
This is not the standard definition of ultrafilter: the usual one requires only that $\emptyset\notin\calU$ instead of property $(1)$, and the ultrafilters not containing finite sets are usually called non-principal ultrafilters. However, we will only need non-principal ultrafilters.
\par
Fix any infinite set $I$. An example of filter on $I$ is the collection of complements of finite sets. An easy application of Zorn's Lemma shows that there actually exists an ultrafilter $\calU$, which extends the mentioned filter. Fix such an ultrafilter. We are ready to define $\st\R$.

\begin{defn}
Define the following equivalence relation $\sim$ on $\R^I=\{f:I\to\R\}$:
$$f\sim g\iff \{i\in I:f(i)=g(i)\}\in\calU.$$
Let $\st\R$ be the quotient set of $\R^I$ modulo this relation.
\end{defn}

It is easily seen using the properties of an ultrafilter (in fact, of a filter) that $\sim$ is indeed an equivalence relation. We can define the sum and the product on $\st\R$ componentwise, as this is easily seen to be well defined. Using also property $(4)$, we obtain that $\st\R$, equipped with this operations, is a field. We can also define an order $\st\leq$ on $\st\R$ in the following way:
$$[f]\ \st\leq [g]\iff \{i\in I:f(i)\leq g(i)\}\in\calU.$$
Using the properties of ultrafilters it is easily seen that this is a total order on $\st\R$ (property $(4)$ is required only to show that it is total), and that $\st\R$ is an ordered field. An embedding of ordered fields $\R\hookrightarrow\st\R$ can be defined simply by $r\mapsto f_r$, where $f_r$ is the function with constant value $r$. We can identify $\R$ with its image in $\st\R$.
\par
Notice that in the definition we gave of $\st\R$ we can substitute $\R$ with any set $X$. Doing so, we obtain the definition of $\st X$, which can be considered as an extension of $X$, just as we considered $\st\R$ as an extension of $\R$. In the case of $\st\R$, we showed that this extension preserves the basic properties of $\R$, i.e. being an ordered field. The idea is that this is true in general, as we will see.
\par
Before proceeding, notice that if $f:X\to Y$ is any function, we can define componentwise a function $\st f:\st X\to\st Y$ (which is well defined), called the nonstandard extension of $f$. This function coincides with $f$ on (the subset of $\st X$ identified with) $X$. Also relations have nonstandard extensions (see the definition of $\st\leq$). Let us give another definition (in a quite informal way), and then we will see which properties are preserved by nonstandard extensions.

\begin{defn}
A formula $\phi$ is bounded if all quantifiers appear in expressions like $\forall x\in X$, $\exists x \in X$ (bounded quantifiers).
\par
The nonstardard interpretation of $\phi$, denoted $\st\phi$, is obtained by adding $\st$ before any set, relation or function (not before quantified variables).
\end{defn}

An example will make these concepts clear: consider
$$\forall X\subseteq \N, X\neq\emptyset\ \exists x\in X\ \forall y\in X\ x\leq y,$$
which expresses the fact that any non-empty subset of $\N$ has a minimum. This formula is not bounded, because it contains ``$\forall X\subseteq \N$''. However, it can be turned into a bounded formula by substituting ``$\forall X\subseteq \N$'' with ``$\forall X\in\calP(\N)$''. The nonstandard interpretation of the modified formula reads
$$\forall X\in\st\calP(\N), X\neq\st\emptyset\ \exists x\in X\ \forall y\in X\ x\st\leq y.\ \ \ \ \ (1)$$
These definitions are fundamental for the theory of nonstandard extensions in view of the following theorem, which will be referred to as the transfer principle.
\begin{thm}\label{transf:thm}
(\L o\v{s} Theorem) Let $\phi$ be a bounded formula. Then $\phi\iff\st\phi$.
\end{thm}

This theorem roughly tells us that the nonstandard extensions have the same properties, up to paying attention to state these properties correctly (for example, replacing ``$\forall X\subseteq \N$'' with ``$\forall X\in\calP(\N)$''). Easy consequences of this theorem are, for example, that the nonstandard extension $(\st G,\st\cdot)$ of a group $(G,\cdot)$ is a group, or that the nonstandard extension $(\st X,\st d)$ of a metric space $(X,d)$ is a $\st\R-$metric space (that is $\st d:\st X\times\st X\to\st\R$ satisfies the axioms of distance, which make sense as $\st\R$ is in particular an ordered abelian group). To avoid too many $\st$'s, we will often drop them before functions or relations, for example we will denote the ``distance'' on $\st X$ as above simply by ``$d$''.
In view of the transfer principle, the following definition is very useful:

\begin{defn}
$A\subseteq\st X$ will be called \emph{internal} subset of $X$ if $A\in\st\calP(X)$. An internal set is an internal subset of some $\st X$.
\par
$f:\st X\to\st Y$ will be called \emph{internal} function if $f\in\st(Y^X)=\st\{f:X\to Y\}$.
\end{defn}

One may think that ``living inside the nonstandard world'' one only sees internal sets and functions, and therefore, by the transfer principle, one cannot distinguish the standard world from the nonstandard world.
\par
Notice that $\st \calP(X)\subseteq \calP(\st X)$ by the transfer principle applied to the formula
$$\forall A\in\calP(X)\ \forall a\in A\ a\in X.$$
Also, $\{\st A: A\in\calP(X)\}\subseteq \st \calP(X)$, by the transfer principle applied to $(\forall a\in A,\ a\in X)\Rightarrow A\in\calP(X)$. To sum up
$$\{\st A: A\in\calP(X)\}\subseteq \st \calP(X)\subseteq \calP(\st X).$$
Analogously, in the case of maps we have
$$\{\st f: f\in Y^X\}\subseteq \st (Y^X) \subseteq (\st Y)^{\st X}.$$
However, the equalities are in general not true, as we will see.
\par
Another example: the transfer principle applied to formula $(1)$, which tells that each non-empty subset of $\N$ has a minimum, gives that each \emph{internal} non-empty subset of $\st X$ has a minimum ($\st\emptyset=\emptyset$ as, for each set $A$, $\exists a\in A\iff \exists a\in\st A$).
\par
\L o\v{s} Theorem alone is not enough to prove anything new. In fact, it holds for the trivial extension, that is, if we set $\st X=X$, $\st f=f$ and $\st R=R$ for each set $X$, function $f$ and relation $R$. However, the nonstandard extensions we defined enjoy another property, which will be referred to as $\aleph_0-$saturation, or simply saturation. First, a definition, and then the statement.

\begin{defn}
A collection of sets $\{A_j\}_{j\in J}$ has the finite intersection property (FIP) if for each $n\in\N$ and $j_0,\dots,j_n\in J$, we have $A_{j_0}\cap\dots\cap A_{j_n}\neq\emptyset$.
\end{defn}

\begin{thm}\label{sat:thm}
Suppose that the collection of internal sets $\{A_n\}_{n\in\N}$ has the FIP. Then $\bigcap_{n\in\N} A_n\neq \emptyset$.
\end{thm}

Let us use this theorem to prove that $\st\R$ contains infinitesimals. It is enough to consider the collection of sets $\{\st (0,1/n)\}_{n\in\N^+}$ and apply the theorem to it. Notice that for $n\in\N^+$, $\st (0,1/n)\in\st\calP(\R)$ as it is of the form $\st A$ for $A\in\calP(\R)$.
More in general, however, for each $x,y\in\st\R$, $(x,y)\in\st\R$ (we should use a different notation for intervals in $\R$ and intervals in $\st\R$, but hopefully it will be clear from the context which kind of interval is under consideration). In fact, we can apply the transfer principle to the formula $\forall x,y\in \R\ (x,y)\in \calP(\R)$. To be more formal, ``$(x,y)\in \calP(\R)$'' should be substituted by
$$\exists A\in\calP(\R)\ \forall z\in\R\ (z\in A\iff x<z{\rm \ and\ }z<y).$$
\par
Notice that it can be proved similarly that $\st\N$ and $\st\R$ contain infinite numbers. We will need the following refinement of this:

\begin{lemma}\label{cof:lem}
\begin{itemize}
\end{itemize}
\begin{enumerate}
\item
Let $\{\xi_n\}_{n\in\N}$ be a sequence of infinitesimals. There exists an infinitesimal $\xi$ greater than any $\xi_n$.
\item
Let $\{\rho_n\}_{n\in\N}$ be a sequence of positive infinite numbers (in $\st\R$ or $\st\N$). There exists an infinite number $\rho$ smaller than any $\rho_n$.

\end{enumerate}

\end{lemma}

\begin{proof}
Let us prove $(1)$, the proof of $(2)$ being very similar.
\par
The collection $\{(\xi_n,1/(n+1)) \}_{n\in\N}$ of internal subsets of $\st\R$ has the FIP. An element $\xi\in\bigcap (\xi_n,1/(n+1))$ has the required properties.
\end{proof}

\begin{conv}
The definition of the nonstandard extensions depends on the infinite set $I$ and the ultrafilter $\calU$. From now on we set $I=\N$ and we fix an ultrafilter $\calU$ on $\N$, and we will consider the nonstandard extensions constructed from this data.
\end{conv}

\par
The reader is suggested to forget the definition of nonstandard extensions, as Theorem~\ref{transf:thm}, Theorem~\ref{sat:thm} and the remark below are all we need, and the definition will never be used again.
\begin{rem}
The nonstandard extension of a set of cardinality at most $2^{\aleph_0}$ has cardinality at most $2^{\aleph_0}$ (this is a consequence of the fact that we set $I=\N$).

\end{rem}

Now, a lemma which is frequently used when working with nonstandard extensions, usually referred to as overspill.

\begin{lemma}
Suppose that the internal subset $A\subseteq\st\R^+$ (or $A\subseteq\st\N$) contains, for each $n\in\N$, an element greater than $n$. Then $A$ contains an infinite number.

\end{lemma}

\begin{proof}
The collection of internal sets $\{A\}\cup\{(n,+\infty)\}_{n\in\N}$ has the FIP, therefore $\bigcap_{n\in\N}(n,\infty)\cap A\neq\emptyset$. (For clarity, here by $(n,\infty)$ we mean $\{x\in\st\R:x>n\}$.) An element in the intersection is what we were looking for.

\end{proof}

We will also need:

\begin{lemma}\label{intsubsX:lem}
Suppose that $A\subseteq X\subseteq \st X$ is internal. Then it is finite.
\end{lemma}

Let us introduce some (quite intuitive) notation, which we are going to use from now on.

\begin{defn}
Consider $\xi,\eta\in\st\R$, with $\eta\neq 0$. We will write:
\begin{itemize}
\item
$\xi\in o(\eta)$ (or $\xi\ll\eta$ if $\xi,\eta$ are nonnegative) if $\xi/\eta$ is infinitesimal,
\item
$\xi\in O(\eta)$ if $\xi/\eta$ is finite,
\item
$\xi\gg\eta$ if $\xi,\eta$ are nonnegative and $\xi/\eta$ is infinite,
\end{itemize}

For example, $o(1)$ is the set of infinitesimals, and $O(1)=\{\xi\in\st\R:|\xi|<r \textrm{\ for\ some }r\in\R^+\}$.

\end{defn}

The map we given by the following lemma plays a fundamental role in nonstandard analysis, and will be used in the definition of asymptotic cone:

\begin{prop}
There exists a map $st:O(1)\to\R$ such that, for each $\xi\in\st\R$, $\xi-st(\xi)$ is infinitesimal.

\end{prop}

We will call $st(\xi)$ the standard part of $\xi$. Notice that $st(\xi)=0\iff \xi$ is infinitesimal.
\par

Many common definitions have interesting nonstandard counterparts. Here is an example which will be used later.

\begin{prop}\label{compactnonstchar:prop}
The metric space $X$ is compact if and only if for each $\xi\in \st X$ there exists $x\in X$ such that $d(x,\xi)\in o(1)$.
\end{prop}

\section{Asymptotic cones}

Let $(X,d)$ be a metric space. The asymptotic cones of $X$ are ``ways to look at $X$ from infinitely far away''. Let us make this idea precise.
\par
First, let us fix some (standard) notation. If $Z$ is a set and $\sim$ is an equivalence relation on $Z$, the equivalence class of $z\in Z$ will be denoted by $[z]$ and the quotient set of $Z$ will be denoted by $Z/\sim$.

\begin{defn}
Consider $\nu\in\st\R$, $\nu\gg 1$. Define on $\st X$ the equivalence relation $x\sim y\iff d(x,y)\in o(\nu)$.
The asymptotic cone $C(X,p,\nu)$ of $X$ with basepoint $p\in\st X$ and scaling factor $\nu$ is defined as
$$\{[x]\in\st X/\sim: d(x,p)\in O(\nu)\}.$$
The distance on $C(X,p,\nu)$ is defined as $d([x],[y])=st\left(\st d(x,y)/\nu\right)$.
\end{defn}

This definition of asymptotic cone is basically due to van den Dries and Wilkie, see~\cite{vDW}.

Before proceeding, a few definitions. If $q\in\st X$ and $d(p,q)\in O(\nu)$, so that $[q]\in C(X,p,\nu)$, then $[q]$ will be called the projection of $q$ on $C(X,p,\nu)$. Similarly, if $A\subseteq \{x\in\st X: d(x,p)\in O(\nu)\}$, the projection of $A$ on $C(X,p,\nu)$ is $\{[a]|a\in A\}$.
\par
The following properties of asymptotic cones are well-known:

\begin{lemma}\label{propcone:lem}
\begin{enumerate}
\item
Any asymptotic cone is a complete metric space.
\item
Any asymptotic cone of a geodesic metric space is a geodesic metric space.
\end{enumerate}
\end{lemma}

\subsection{Asymptotic cones of groups}

\begin{conv}
All groups from now on are implied to be finitely generated. 
\end{conv}

Consider a group $G$ and a finite generating set $S=S^{-1}$ for $G$ (even if not explicitly stated, we will always assume that generating sets satisfy $S=S^{-1}$). We are going to define a metric graph $\mathcal{CG}_S(G)$ associated to $(G,S)$, which is called the Cayley graph of $G$ with respect to $S$.
The vertices of $\mathcal{CG}_S(G)$ are the elements of $G$, and there is an edge of length 1 between $g$ and $h$ if and only if there exists $s\in S$ such that $gs=h$. We will consider $\mathcal{CG}_S(G)$ endowed with the path metric induced by this data.

It is easy to prove that $\mathcal{CG}_S(G)$ is a geodesic metric space.
\par
We are now ready to define the asymptotic cone of a group.

\begin{defn}
Let $G$ be a finitely generated group and $S$ a finite generating set for $G$. The asymptotic cone $C_S(G,g,\nu)$ of $G$ with basepoint $g\in\st G$ and scaling factor $\nu\gg 1$ is $C(\mathcal{CG}_S(G),g,\nu)$.

\end{defn}

Let us state some useful and well-known properties of asymptotic cones of groups.

\begin{lemma}\label{propconegroups:lem}
For any group $G$, finite generating sets $S,S'$, $g,g'\in\st\G$, $\nu\gg 1$:

\begin{enumerate}
\item
$C_S(G,g,\nu)$ is complete, geodesic and homogeneous,
\item
$C_S(G,g,\nu)$ is isometric to $C_S(G,g',\nu)$,
\item
$C_S(G,g,\nu)$ is $k-$bilipschitz homeomorphic to $C_{S'}(G,g,\nu)$.
\end{enumerate}

\end{lemma}

Notice that the third property implies that the topological properties of the asymptotic cones of a group do not depend on the choice of a finite generating set. This is why asymptotic cones are useful to study groups.

\section{Use of nonstandard methods for asymptotic cones}
The main aim of this section is to show how nonstandard methods can be used to prove results about asymptotic cones.
\par
We will see that the following lemma gives several obstructions for a space to be realized as an asymptotic cone.
\begin{lemma}\label{finuncount:lem}
An internal set is finite or has cardinality at least $2^{\aleph_0}$.
\end{lemma}

\begin{proof}
Any set is finite or admits an injective function from $\N$. By the transfer of this property, we have that every internal set admits a bijective (internal) function from $\{0,\dots,\nu\}$ for some $\nu\in\, ^*\N$ or an injective (internal) function from $^*\N$.
\par
So it is enough to prove that the set $\{0,\dots,\nu\}$ is uncountable for every infinite $\nu$. The fact that the map
$$\alpha\in \{0,\dots,\nu\}\mapsto st(\alpha/\nu)\in [0,1]$$
is surjective implies the claim.

\end{proof}

Let $X$ be a metric space. For $p\in X$ and $r_1,r_2,l\geq 0$ denote by $F_X(p,r_1,r_2,l)$ the supremum of the cardinalities of sets $M$ satisfying
\begin{enumerate}
\item
$\forall x\in M, \ r_1\leq d(x,p)\leq r_2,$
\item
$\forall x,y\in M,\ x\neq y\Rightarrow d(x,y)\geq l$.
\end{enumerate}
A set $M$ satisfying the above properties will be called, for $\alpha\leq |M|$, a test for $F_X(p,r_1,r_2,l)\geq \alpha$.

\begin{figure}[h]
\centering
\includegraphics[height=5.5cm]{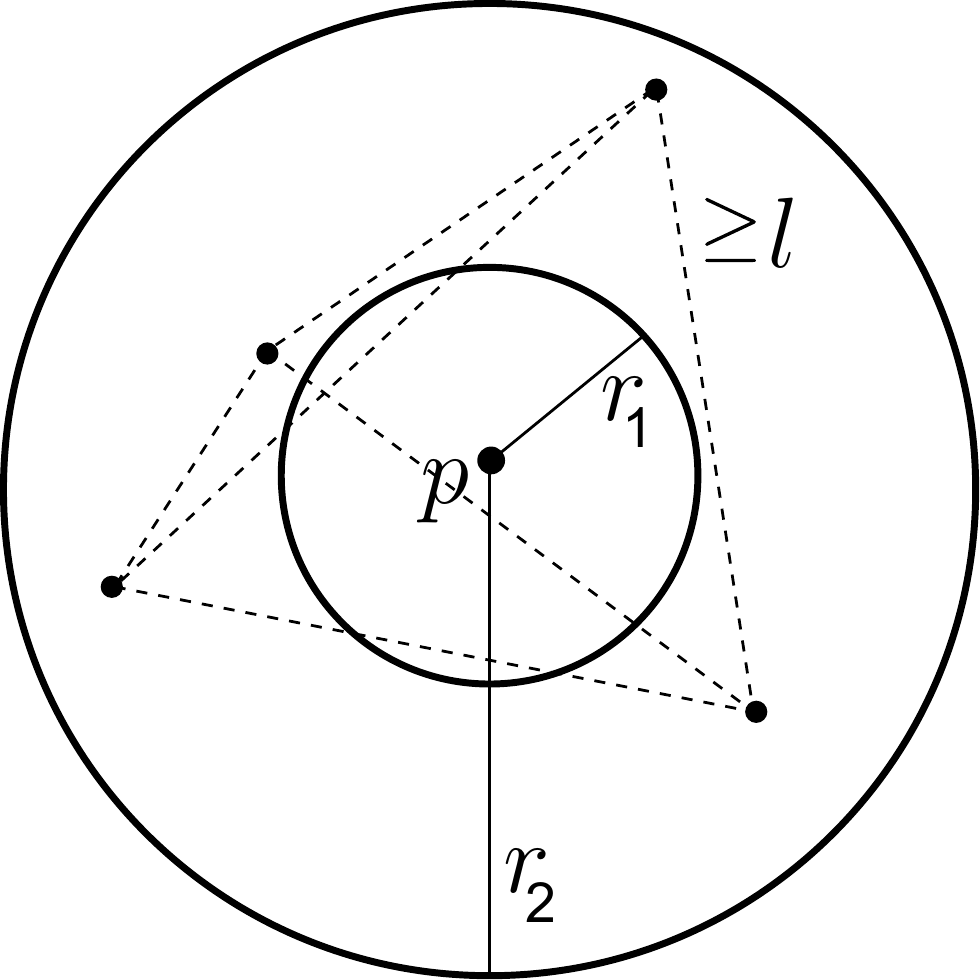}
\end{figure}

\begin{rem}
\begin{itemize}
\item
If $F_X(p,r_1,r_2,l)$ is finite, then it is a maximum,
\item
for each $\alpha<F_X(p,r_1,r_2,l)$ we can find a test for $F_X(p,r_1,r_2,l)\geq\alpha$.
\end{itemize}

\end{rem}

For convenience, if $\rho\in\st\N$ is infinite, we set $st(\rho)=2^{\aleph_0}$. We also set $st(\alpha)=2^{\aleph_0}$ for each $\st$cardinality $\alpha\geq\st\aleph_0$.
\par 
We will often use the following easy properties:

\begin{rem}\label{properties:rem}
\begin{enumerate}
\item
$F_X(p,r_1,r_2,l)\leq F_X(p',r'_1,r'_2,l')$ if and only if for each test $M$ for $F_X(p,r_1,r_2,l)\geq\alpha$ there is a test $M'$ for $F_X(p',r'_1,r'_2,l')\geq\alpha$,
\item
if $r_1\geq r'_1$, $r_2\leq r'_2$ and $l\geq l'$ then $F_X(p,r_1,r_2,l)\leq F_X(p,r'_1,r'_2,l')$,
\item
if $X$ is the asymptotic cone of $Y$ with basepoint $p$ and scaling factor $\nu$, then $st\left(F_{\st Y}(p,\rho_1\nu,\rho_2\nu,\lambda\nu)\right)\leq F_X([p],st(\rho_1),st(\rho_2),st(\lambda))$, where $\rho_1,\rho_2$ and $\lambda$ are finite and $st(\lambda)>0$.

\end{enumerate}
\end{rem}

\begin{proof}
$(1)$ is straightforward from the definitions.
\par
$(2)$ If $M$ is a test for $F_X(p,r_1,r_2,l)\geq \alpha$, then it is also a test for $F_X(p,r'_1,r'_2,l')\geq\alpha$.
\par
$(3)$ By our convention on the standard part and the fact that $|\st\N|=2^{\aleph_0}$, $st(\nu)$ is the cardinality of $\nu=\{0,\dots,\nu-1\}$. In fact, if $\nu$ is finite, $st(\nu)=\nu$, otherwise $st(\nu)=2^{\aleph_0}\leq |\nu| \leq |\st \N|=2^{\aleph_0}$ (we used Lemma~\ref{finuncount:lem}). If $M$ is a test for $F_{\st Y}(p,\rho_1\nu,\rho_2\nu,\lambda\nu)\geq\alpha$, its projection on $X$ is a test for $F_X([p],st(\rho_1),st(\rho_2),st(\lambda))\geq st(\alpha)$ (projections of distinct elements of $M$ are distinct because their distance is at least $st(\lambda)>0$).

\end{proof}

\begin{prop}\label{uncount:prop}
Let $X$ be a metric space. If $X$ is an asymptotic cone then for each $p,r_1,r_2,l$, with $l>0$, if $F_X(p,r_1,r_2,l)$ is infinite then it is at least $2^{\aleph_0}$.

\end{prop}

\begin{proof}
Assume that $X$ is an asymptotic cone of $Y$, with scaling factor $\nu$, and fix $p,r_1,r_2,l$ as above and such that $F_X(p,r_1,r_2,l)\geq \aleph_0$. Fix a representative $\pi\in\st Y$ of $p$. For each $n$, one can find $x_{n,1},\dots,x_{n,n}\in \st Y$ such that
\begin{itemize}
\item
each $x_{n,i}$ is at a distance $(r_{n,i}+\xi_{n,i})\nu$ from $\pi$, for some $r_{n,i}\in[r_1,r_2]$ and some infinitesimal $\xi_{n,i}$,
\item
for each $n$ and $i\neq j$, $d(x_{n,i},x_{n,j})>(l-\xi_{n,i,j})\nu$ for some positive infinitesimal $\xi_{n,i,j}$.
\end{itemize}
\par
We can bound all the $|\xi_{n,i}|$'s and $\xi_{n,i,j}$ by some positive infinitesimal $\xi$, by Lemma~\ref{cof:lem}.
So we have that $F_{\st Y}(\pi,(r_1-\xi)\nu,(r_2+\xi)\nu,(l-\xi)\nu)$ is greater than any finite $n$, hence it is greater than some infinite $\rho\in\st\N$. The conclusion follows from point $(3)$ of Remark~\ref{properties:rem}.

\end{proof}

We will now study the consequences of this proposition for real trees appearing as asymptotic cones. These are interesting objects in view of the fact that a geodesic metric space is hyperbolic if and only if each of its asymptotic cones is a real tree (see for example~\cite{Dr1} or~\cite{FS}). Groups such that at least one of their asymptotic cones is a real tree are studied in~\cite{OOS}.

\begin{cor}\label{nocount:cor}
If $X$ is a real tree such that every geodesic can be extended (e.g.: a homogeneous real tree, see~\cite{DP}) and the valency at a point $p$ is infinite, then this valency is at least $2^{\aleph_0}$.
\end{cor}

\begin{proof}
Our assumption on geodesics implies that, for each $r>0$, $F_X(p,r,r,2r)$ equals the valency at $p$.

\end{proof}

\begin{defn}
In a real tree, a point of valency greater than 2 will be called a branching point.
\end{defn}

\begin{prop}\label{isolated:prop}
Let $X$ be a real tree such that each geodesic can be extended and the valency at a point $p$ is finite. If $X$ is an asymptotic cone then $p$ is isolated from the other branching points.

\end{prop}

\begin{proof}
Let $n$ be the valency of $X$ at $p$. For each $r>0$, $F_X(p,r,r,2r)=n$. Assume that $p$ is not isolated from the other branching points. Then for each $k\in\N$ (and $k>1/2r$) we have that $F_X(p,r,r,2r-1/k)$ is infinite.
If $X$ is an asymptotic cone of $Y$ with scaling factor $\nu$ and $\pi\in\st Y$ is a representative for $p$, proceeding as in the proof of Proposition~\ref{uncount:prop}, for each $k$ we can find a positive infinitesimal $\xi_k$ and a positive infinite $\mu_k$ such that $F_{\st Y}(\pi, (r-\xi_k)\nu,(r+\xi_k)\nu, (2r-1/k-\xi_k)\nu)\geq\mu_k$.
Let us fix a positive infinitesimal $\xi$ greater than any $\xi_k$ and a positive infinite $\mu$ smaller than any $\mu_k$. We have that $\{\alpha|F_{\st Y}(\pi, (r-\xi)\nu,(r+\xi)\nu, (2r-\alpha)\nu)\geq \mu\}$ contains, for each $k$, elements of $\st \R$ smaller than $1/k$ (for example $1/(k+1)+\xi_{k+1}$), hence it contains an infinitesimal $\eta$. This implies that $F_X(p,r,r,2r)$ is infinite (using point $(3)$ of Remark~\ref{properties:rem}), a contradiction.

\end{proof}

Putting together Corollary~\ref{nocount:cor} and Proposition~\ref{isolated:prop} in the case of homogeneous real trees, we have:

\begin{cor}
If $X$ is a homogeneous real tree and an asymptotic cone, then it is a point, a line or it has valency at least $2^{\aleph_0}$ at each point.
\end{cor}

As Cayley graphs of groups have cardinality $2^{\aleph_0}$, their asymptotic cones have cardinality at most $2^{\aleph_0}$ and hence:

\begin{cor}~\label{3trees:cor}
If the real tree $X$ is the asymptotic cone of a group, then it is a point, a line or a tree with valency $2^{\aleph_0}$ at each point.
\end{cor}

\par
In view of the following theorem, proved in~\cite{DP}, this corollary implies that there are 3 possible isometry types of real trees appearing as asymptotic cones of groups.

\begin{thm}
If $T_1, T_2$ are homogeneous real trees such that the valency at a point in $T_1$ is the same as the valency at a point in $T_2$, then $T_1$ and $T_2$ are isometric.
\end{thm}

Now, let us analyze the consequences of Proposition~\ref{uncount:prop} in the special case $r_1=0$, proving Theorem~\ref{proper:thm} which we restate for the convenience of the reader.

\begin{thm}
\begin{enumerate}
\item
If the separable metric space $X$ is an asymptotic cone, then $X$ is proper.
\item
Suppose that for some $\mu\ll\nu$ and some $p\in\st Y$ each asymptotic cone of the metric space $Y$ with scaling factor $\nu'$ such that $\mu\leq\nu'\leq\nu$ and basepoint $p$ is separable and that $X=C(Y,p,\nu)$ is homogeneous. Then $X$ has finite Hausdorff and Minkowski dimension.
\end{enumerate}
\end{thm}

\begin{proof}
$(1)$ Let $X$ be the asymptotic cone of $Y$ with basepoint $p$ and scaling factor $\nu$. Suppose that $B=\overline{B}([p],r)\subseteq X$ is not compact. Then, as $X$ and therefore $B$ is complete (by Lemma~\ref{propcone:lem}-(1)), there exists some $0<\epsilon<1$ such that $B$ cannot be covered by finitely many balls of radius $\epsilon r$. Set $r'=\epsilon r/2$. Suppose by contradiction that $F_{\st Y}(p,0,(r+1)\nu, r'\nu)$ is finite (say equal to $n$), consider a test $M\subseteq \st Y$ of $F_{\st Y}(p,0,(r+1)\nu, r'\nu)\geq n$. Let $\{x_1,\dots,x_n\}\subseteq X$ be the projection of $M$ onto $X$ (the $x_i$'s are distinct because $r'>0$).
As, by our hypothesis, $B$ cannot be a subset of $\bigcup B(x_i,\epsilon r)$, we can find $[y]\in B\backslash\bigcup B(x_i,\epsilon r)$. As $r'<\epsilon r$, it is easily seen that $M\cup \{y\}$ is a test for $F_{\st Y}(p,0,(r+1)\nu, r'\nu)\geq n+1$, which is a contradiction: in fact, $d(p,y)<(r+1)\nu$ as $st(d(p,y)/\nu)=r$, and, for each $m\in M$, $d(y,m)\geq r'\nu$ as $st(d(p,m)/\nu)\geq \epsilon r>r'$.
\par
Therefore, $F_{\st Y}(p,0,(r+1)\nu, r'\nu)\geq \rho$, for some infinite $\rho\in\st\N$.
Let us show that this implies that $B([p],r+1)$ is not separable, as it should be as $X$ is separable (subsets of separable \emph{metric} spaces are separable). We have that there exists a test $M$ for $F_X([p],0,r+1,r')\geq 2^{\aleph_0}$, which is obtained projecting onto $X$ a test for $F_{\st Y}(p,0,(r+1)\nu, r'\nu)\geq \rho$. If $m_1,m_2\in M$ are distinct, we have $B(m_1,r'/2)\cap B(m_2,r'/2)=\emptyset$ as $d(m_1,m_2)\geq r'$. This gives an uncountable family of non-intersecting balls whose centers lie in $B([p],r+1)$. The existence of such a family clearly implies that $B([p],r+1)$ is not a separable space.
\par
$(2)$ We will prove that there exists $n\in\N$ such that for each $r$ with $0<r\leq 1$ we have that $B([p],r)$ can be covered by at most $n(r)$ balls of radius $r/2$.
\par
By $(1)$, we know that fixing $r$ we can choose $n(r)$ such that $B([p],r)$ can be covered by at most $n$ balls of radius $r/2$. Suppose by contradiction that $n(r)$ is not bounded when $r\to 0$. Notice that for any radius $r>0$ we have that $F_{\st Y}(p,0,2r\nu, r\nu/8)$, if finite, provides an upper bound for $n(r)$. In fact, consider a maximal test $M$. For each $m\in M$ pick $p(m)\in X$ such that $p(m)\in B([m],r/4)\cap \overline{B}([p],r)$, if such $p(m)$ exists. We have that $\overline{B}([p],r)$ is contained in $\bigcup B(p(m),r/2)$, because if $[x]\in \overline{B}([p],r)\backslash \left(\bigcup B(p(m),r/2)\right)$, then it is easily seen that $M\cup \{x\}$ is a test for $F_{\st Y}(p,0,2r\nu, r\nu/8)\geq |M|+1$.
\par
So, we have that $F_{\st Y}(p,0,2r\nu, r\nu/8)$ is not bounded by any natural number for $r\to 0$. In particular, we can find an infinitesimal $\xi$ such that $F_{\st Y}(p,0,2\xi\nu,\xi\nu/8)\geq \rho$ for some infinite $\rho$, and we can also choose it so that $\xi\nu\geq\mu$. Using essentially the same argument as in the last part of $(1)$ we get that $C(Y,p,\xi\nu)$ is not separable, a contradiction.

\end{proof}

\begin{cor}\label{sepvirtnilp:cor}
If there exist $\nu_1\ll\nu_2$ such that all the asymptotic cones of the group $G$ with scaling factor $\nu\in[\nu_1,\nu_2]$ are separable, then $G$ is virtually nilpotent.

\end{cor}

\begin{proof}
Using the theorem above we get an asymptotic cone of $G$ which is proper and has finite Minkowski dimension, and then we can directly apply the main result of~\cite{Po}.

\end{proof}

\begin{rem}
The corollary above answers the question (asked in~\cite[Remark 6.4-(3)]{vDW}) whether or not local compactness of one asymptotic cone of a group implies that the group is of polynomial growth, even if the proof is far from being direct. 

\end{rem}

Theorem~\ref{proper:thm} provides many examples of metric spaces which do not appear as asymptotic cones, for example the separable Hilbert space. In contrast, we will prove below a ``positive'' result on spaces which are realized as asymptotic cones.

\begin{thm}\label{positive:thm}
If the metric space $X$ is proper, then it is an asymptotic cone of some metric space $Y$. If $X$ is also geodesic and unbounded, we can choose $Y$ to be geodesic as well.
\end{thm}

Notice that if $X$ is not geodesic, then $Y$ cannot be geodesic by Lemma~\ref{propcone:lem}.
\par
The first part of the statement above has been proven independently by Scheele in~\cite{Sc}, using a different construction. Our construction is a slight variation of the one which appears in~\cite[Section 5]{FS}, translated in the nonstandard setting.

\begin{proof}
Let us first assume that $X$ is unbounded. Let $\{p_n\}$ be a sequence of points of $X$ such that $d(p_0,p_n)\to\infty$. Set $Y=(X\times \N)\cup \bigcup_{n\in\N} (\{p_n\}\times[n,n+1])$. Define a distance $\widetilde{d}$ on $Y$ in the following way:
$$
\widetilde{d} ((x,t),(x',t'))=
\left\{
\begin{array}{lll}
t\cdot d(x,p_t) + t'\cdot d(p_{t'},x') + |t-t'| & {\rm if} & t\neq t'\\
t\cdot d(x,x') & {\rm if} & t=t'
\end{array}
\right.
$$
It is quite clear that $Y$ is a metric space, and that it is geodesic if $X$ is geodesic. Consider now $\st Y=(\st X\times \st \N)\cup \bigcup_{\mu\in\st\N} (\{p_{\mu}\}\times[\mu,\mu+1])$, and an infinite $\nu\in\st\N$. We want to show that the asymptotic cone $Z$ of $Y$ with basepoint $(p_0,\nu)$ and scaling factor $\nu$ is isometric to $X$. The isometry $i:X\to Z$ can be defined simply by $x\mapsto [(x,\nu)]$. It is readily checked that it is an isometric embedding. So far we did not use properness or that $d(p_0,p_n)\to\infty$, so we obtained the following.
\begin{rem}
Any metric space $X$ can be isometrically embedded in an asymptotic cone of a metric space $Y$. If $X$ is geodesic, we can require $Y$ to be geodesic.
\end{rem}
Section 5 of~\cite{FS} already contains a proof of this fact.
\par
We are left to prove that $i$ is surjective. First of all, notice that the distance of any element of $\st Y\backslash (\st X\times\{\nu\})$ from $(p_0,\nu)$ is at least
$$\min\{\nu d(p_0,p_{\nu-1}), \nu d(p_0,p_{\nu+1})\}\gg \nu,$$
as $d(p_0,p_n)\to\infty$ and so $d(p_0,p_\mu)\gg 1$ for each infinite $\mu\in\st\N$. Therefore no element of $\st Y\backslash (\st X\times\{\nu\})$ projects onto an element of $Z$. What remains to prove is that for each $y\in\st X$ with $\widetilde{d}((p_0,\nu),(y,\nu))/\nu=d(p_0,y)\in O(1)$ there exists $x\in X$ such that $\widetilde{d}((x,\nu),(y,\nu))/\nu=d(x,y)\ll 1$. Consider $y$ as above and some $r>d(p_0,y)$, $r\in\R$. We have that $B=\overline{B}_X(p_0,r)$ is compact and $y\in\st B$. By the nonstandard characterization of compact metric spaces (Proposition~\ref{compactnonstchar:prop}), there exists $x\in B$ such that $d(x,y)\ll 1$, and we are done.
\par
The case that $X$ is bounded can be handled similarly. Fix $p\in X$ and set $Y=X\times\N$. Define
$$
\widetilde{d} ((x,n),(x',n'))=
\left\{
\begin{array}{lll}
n\cdot d(x,p) + n'\cdot d(p,x') + |n^2-(n')^2| & {\rm if} & n\neq n'\\
n\cdot d(x,x') & {\rm if} & n=n'
\end{array}
\right.
$$

Modifying the previous proof, it is easily shown that, for any infinite $\nu\in\st\N$, the asymptotic cone of $Y$ with basepoint $(p,\nu)$ and scaling factor $\nu$ is isometric to $X$.

\end{proof}

\subsection*{Aknowledgments}
The author is greatly indebted to his Master's thesis advisor Roberto Frigerio. He would also like to thank Cornelia Dru\c{t}u for several helpful comments.

\bibliographystyle{model1b-num-names}

\end{document}